\documentclass{amsart}
\usepackage{amsmath}
\usepackage{amsfonts}
\usepackage{amssymb}
\usepackage{cite}
\newcommand{\restrict}{\!\upharpoonright\!}

\newtheorem{theorem}{Theorem}[section]
\newtheorem{lemma}[theorem]{Lemma}
\newtheorem{claim}[theorem]{Claim}

\newtheorem{proposition}[theorem]{Proposition}

\theoremstyle{definition}
\newtheorem{definition}[theorem]{Definition}
\newtheorem{remark}[theorem]{Remark}

\def\sat{\models}

\def\QQ{\mathbb Q}

\def\emb{\preceq}

\DeclareMathOperator\tp{tp}
\DeclareMathOperator\lh{length}

\DeclareMathOperator\acl{acl}
\DeclareMathOperator\dcl{dcl}
\DeclareMathOperator\Ded{Ded}

\DeclareMathOperator\Pri{Pr}
\DeclareMathOperator\Span{span}

\newcommand\iseg{\sqsubset}

\begin{document}
\keywords{o-minimality, model theory, small extensions}
\subjclass[2000]{Primary 03C64; Secondary 03C50}

\title{Maximal small extensions of o-minimal structures}
\author{Janak Ramakrishnan}

\begin{abstract}
A proper elementary extension of a model is called small if it
  realizes no new types over any finite set in the base model.  We answer a
  question of Marker, and show that it is possible to have an o-minimal
  structure with a maximal small extension.  Our construction yields such a
  structure for any cardinality.  We show that in some cases, notably when the
  base structure is countable, the maximal small extension has maximal possible
  cardinality.
\end{abstract}

\maketitle

\section{Introduction}
Kueker, referenced in \cite{HrSh91}, defines:

\begin{definition}
Let $M\prec N$ be models.  $N$ is a \emph{small extension of $M$} if, for any
$a\in N$ and finite $A\subset M$, the type $\tp(a/A)$ is realized in $M$.
\end{definition}

Kueker asked the question: for a general $M$, is there a ``Hanf number''
$\lambda$ such that, if $M$ has a small extension of any cardinality below
$\lambda$, $M$ has small extensions of every cardinality?  Hrushovski and Shelah
\cite{HrSh91} answered this question for superstable $M$ -- in a countable
theory it is $\beth_{\omega_1}$.

It is sometimes more convenient to work with a ``maximal'' small extension -- a
small extension that does not imply the existence of small extensions of every
cardinality, and has maximal cardinality among small extensions like this.  The
existence of a maximal small extension gives the Hanf number as two more than
the cardinality of this extension.  Marker \cite{Marker86} showed that any
maximal small extension of an o-minimal structure could have cardinality at most
$2^{|M|}$.  Marker's argument uses the fact that there are at most $2^{|M|}$
types over $M$, so that an extension of greater cardinality would have to
realize at least one type more than once.  Since there are actually at most
$\Ded(|M|)$ types over $M$, where
\[
\Ded(\alpha)=\sup\{|\bar Q| : Q\text{ a linear order,
}|Q|\le\alpha\}
\]
and $\bar Q$ denotes the completion of the linear order $Q$, Marker's argument
shows that a maximal small extension must have cardinality at most $\Ded(|M|)$.

Recently, Kudaibergenov \cite{Kudaibergenov8} has shown that for weakly
o-minimal theories with atomic models (a class that includes o-minimal
theories), the above result can be tightened to say that a maximal small
extension must have cardinality at most $2^{|T|}$, where $|T|$ is the cardinality
of the theory.

Most examples of o-minimal structures either have no small extensions or
unboundedly many -- in a pure dense linear order, every extension is small.  In
the rationals as an ordered group, no extension is small.  In fact, no
non-trivial examples of o-minimal structures with both small and non-small
extensions were known (obvious ``glueing'' examples can be constructed).

In this paper, we construct the first known example of an o-minimal structure
with a maximal small extension.  This construction generalizes to construct such
examples for all cardinalities.  In the countable case, our maximal small
extension has cardinality $2^{\aleph_0}$, which is as large as possible.  For
some larger cardinalities, our maximal small extension will have cardinality
equal to the corresponding Dedekind number.  In general, though, we cannot show
optimality, although for example the Generalized Continuum Hypothesis or a
similar assumption would make our maximal small extensions have greatest
possible cardinality.

The results in \cite{Kudaibergenov8} imply that if $M$ is a structure with a
maximal small extension and $|M|=\alpha$ for some cardinal $\alpha$, then the
cardinality of the theory $T$ must be at least large enough that
$2^{|T|}>\alpha$.  Again, for each cardinal $\alpha$ our construction may give
$M$ with minimal $|T|$, subject to set-theoretic considerations.

Thanks to Thomas Scanlon for reading and commenting on an earlier version of
this paper, as well as to Charles Smart for discussions that led to the proof of
the cardinality of $N$.

\section{Types}

Given $p,q\in S(M)$, say that $q$ is \emph{definable} from $p$ if for any $c\sat
p$ there is $d\sat q$ with $d\in\dcl(Mc)$.

\begin{definition} Given a model, $M$, and a set, $A\subseteq M$,
let a type $p\in S_1(M)$ be \emph{$A$-finite} iff for some finite $\bar b\in A$,
the type $p\restrict\bar b$ generates $p$.  The type $p\in S(M)$ is \emph{almost
  $A$-finite} iff there exists an $A$-finite type that is definable from $p$.
\end{definition}

Since order-type implies type in o-minimal theories, $A$-finiteness has an
interpretation in the order -- $\dcl(\bar b)$ is dense in $M$ near $c$ a
realization of $p$, for $\bar b\subseteq A$ the witness to $A$-finiteness.
Considering this interpretation, we have:
\begin{remark}
Let $M$ be o-minimal and $N$ an elementary extension of $M$.  If $N$ realizes no
$M$-finite types then $N$ is a small extension of $M$.
\end{remark}

We recall here a classification of types given in \cite{Marker86}:

\begin{definition} A $1$-type, $p\in S(A)$, $A=\acl(A)$, is \emph{non-principal}
  iff $p$ is not algebraic, $p$ contains formulas of the form $x<a$ and $a<x$,
  and for each formula of the form $a<x$, there is $b\in A$ with $b<x$ in $p$,
  and similarly for $x<a$.  $p$ is \emph{principal} iff it is not algebraic and
  not non-principal.  A non-principal $p$ is \emph{uniquely realizable} if the
  prime model realizing $p$ has just one realization of $p$.  \end{definition}

\section{Existence of Maximal Small Extensions}

\begin{proposition}
For every $\alpha$, there is an o-minimal structure, $M$,
$|M|=\alpha$, with small extensions but not unboundedly large small
extensions.  Moreover, if $\alpha$ is of the form $\beta^{<\lambda}$,
for some $\lambda$, a small extension can be found of cardinality
$\beta^\lambda$.
\end{proposition}

\begin{proof}
We give a construction of models $M$ and $N$, with $M\emb N$, and $N$
a maximal small extension of $M$.  We then verify the sizes of $M$ and
$N$.

Let $G$ be a divisible ordered abelian group, $\lambda$ an ordinal,
$Q$ a dense divisible subgroup of $G$.  Let $Q'=G\setminus Q$.

Let $M=G^{<\lambda}$.  We consider $M$ as a subgroup of $G^\lambda$, which is
ordered lexicographically and equipped with group structure component-wise.  Let
our language be that of an ordered group, extended by constants for every
element of $Q^{<\lambda}$.  We will build $N$ in stages.

\begin{itemize}

\item Let $M_0=M$.

\item Given $M_i$, choose $a\in G^\lambda$ such that any $b\in
  \dcl(aM_i)\setminus M_i$ has cofinal components in $Q'$.  Let
  $M_{i+1}=\Pri(M_ia)$ (the prime model containing $M_i$ and $a$).  If no such
  $a$ exists, then we halt.
\item Take unions at limits.

\end{itemize}

This construction must halt at some point, since there are
$\le|G|^\lambda$ elements to add.  Let the union of the $M_i$'s be $N$.

$M$ is o-minimal, since it is a divisible ordered abelian group, and each $M_i$
and $N$ is an elementary extension, since they are also divisible ordered
abelian groups and this theory has quantifier elimination.

It remains to be shown that $N$ is a small extension of $M$, and that
there is no larger small extension of $M$.  In fact, we show that
every small extension of $M$ comes from this type of construction.

Notation: we use $M'$ to denote an arbitrary $M_i$ or $N$.  For
$\alpha<\lambda$, $a[\alpha]$ is the $\alpha$-th component of $a$, and
$a\restrict \alpha=\langle a[i]\rangle_{i<\alpha}$.  We let $a\iseg b$ denote
``$a$ is an initial segment of $b$.''

\begin{lemma}
Every principal type over $M'$ is almost $M$-finite.
\end{lemma}

\begin{proof}

  Let $p$ be principal over $M'$.  Let $p$ be generated by the formulas
  $\{a<x\}\cup\{x<e\mid e\in M', e>a\}$ (the other cases are similar).  Let $d$
  be any realization of $p$.  The type of $d-a$ over $M'$ is generated by
  $\{0<x\}\cup\{x<e\mid e>0\}$ -- the principal type near $0$.  Given any
  $e>0\in M'$, let $\alpha$ be the first index at which $e[\alpha]\ne 0$.  Let
  $c\in Q^{\alpha+1}$ be such that $c[i] = 0$ for $i<\alpha$, and
  $0<c[\alpha]<e[\alpha]$.  Then $0<c<e$.  Thus, $x<c$ implies $x<e$, and
  $d-a<c$, so $\tp(d-a/M')$ is generated by $\tp(d-a)$.

\end{proof}

\begin{definition} Let $p\in S_1(M')$ be non-principal.  $p$ is \emph{reducible}
  if there is $\alpha<\lambda$ such that, for any $a,b\in M'$, if $a\restrict
  \alpha=b\restrict \alpha$ then $x<a\in p\iff x<b\in p$.
\end{definition}

Note that if $p$ is reducible then it is not uniquely realizable, since for
$\alpha$ as in the definition and any $a$ with $\lh(a)>\alpha$ and
$a\restrict\alpha=0$, if $c\sat p$, then $c+a\sat p$.  

\begin{lemma}
If $p\in S_1(M')$ is reducible, then $p$ is $M$-finite.
\end{lemma}

\begin{proof}

  Let $\alpha$ be the least such in the definition of reducible.  For each
  $\beta<\alpha$, we can find $a_\beta\in M$, $a_\beta^+,a_\beta^-$ extending
  $a_\beta$ such that $\lh(a_\beta)=\beta$, $x<a_\beta^+\in p$, and
  $x>a_\beta^-\in p$.  It is easy to see that $\beta<\beta'$ implies
  $a_\beta\iseg a_{\beta'}$.  Let $a=\bigcup_{\beta<\alpha} a_\beta$.  Then
  $a\in M$.

Let $d$ realize $p$, let $e$ be any element of $M'$.  WLOG, assume
$e>d$.  We show $e>d$ is implied by $\tp(d/a)$.

Case 1: $e\restrict \alpha\ne a\restrict \alpha$.  Then $e$ and $a$ differ at
some coordinate $\beta<\alpha$, so $e[\beta]>a[\beta]$, since
$e>d>a\restrict\beta+1$.

If $a>d$, we are done.  Otherwise, by density of $Q$, we can find
$c\in Q^{\beta+1}$ with $c[i]=0$ for $i<\beta$, and
$a[\beta]<a[\beta]+c[\beta]<e[\beta]$.  Again, it is clear that
$a+c>d$, so we are done for this case.

Case 2: $e\restrict \alpha=a\restrict \alpha$.  Since we assume $e>d$, we also
have $a>d$, since $p$ is reducible.  Let $\beta\ge \alpha$ be the first
coordinate at or past $\alpha$ at which $e$ is not $0$ (if $\beta$ does not
exist then $e=a$).  If $e[\beta]>0$, we are done (since $e>a>d$), so let
$e[\beta]<0$.  Choose $c\in Q^{\beta+1}$ such that $c[i]=0$ for $i<\beta$, and
$c[\beta]<e[\beta]<0$.  Then $c+a<e$, but since $(c+a)\restrict
\alpha=e\restrict \alpha$, that means $c+a>d$, so $c+a>x\in \tp(d/a)$, and hence
$e>d$ is implied by $\tp(d/a)$.

\end{proof}

\begin{lemma}
If $p\in S_1(M')$ is non-reducible, then for some $a\in
G^\lambda$, $\tp(a/M')=p$.
\end{lemma}

\begin{proof}

For each $\alpha<\lambda$, by non-reducibility, there are
$a_\alpha^-,a_\alpha^+\in M'$ such that
$a_\alpha^-\restrict\alpha=a_\alpha^+\restrict\alpha$, but $a_\alpha^-<x<a_\alpha^+\in
p$.

Let $a_\alpha=a_\alpha^-\restrict\alpha$.  It is easy to check that
$\alpha<\alpha'$ implies $a_\alpha\iseg a_{\alpha'}$.

Let $a=\bigcup_{\alpha<\lambda}a_\alpha$.  If $a<e$, then at some
component, say $\alpha$, $a[\alpha]<e[\alpha]$.  But
$a\restrict\alpha+1=a^+_{\alpha+1}$, so $a^+_{\alpha+1}<e$, so $x<e\in
p$.

The case $e<a$ is symmetric.  Thus, $\tp(a/M')=p$.

\end{proof}

Given a sequence, $a$, we say that ``$a$ has cofinal components with property
$P$'' if for any $\lambda<\lh(a)$, there is $\kappa>\lambda$ such that
$a[\kappa]$ has property $P$.

\begin{lemma}
Let $d\in G^\lambda$ realize a non-reducible type over $M'$ without
cofinal components in $Q'$.  Then $\tp(d/M')$ is $M$-finite.
\end{lemma}

\begin{proof}

  For some $m<\lambda$, $b=d\restrict m$ has all the components of $d$ in $Q'$.
  Note that $b\in M$.  Given any $e\in M'$ with $x<e\in\tp(d/M')$, let $n$ be
  the first index at which $d$ and $e$ differ.

If $n<m$, let $c\in Q^{n+1}$ be such that $c[i]=0$ for $i<n$, and
$0<c[n]<e[n]-b[n]$.  Then $x<b+c$ is in $\tp(d/b)$, and $b+c<e$.

If $n\ge m$, then choose $c\in Q^{n+1}$ such that $c[i]=0$ for $i<m$,
$c[i]=a[i]$ for $m\le i<n$, and $d[n]<c[n]<e[n]$.  Then $x<b+c$ is in $\tp(d/b)$
and $b+c<d$.

The $e<x$ case is symmetric.

\end{proof}

\begin{lemma} If $d\in G^\lambda\setminus M'$ has cofinal components in $Q'$ ,
  then $\tp(d/M')$ is not $M$-finite.  Thus, if every $b\in\dcl(dM')\setminus
  M'$ has cofinal components in $Q'$, then $d$ is not almost $M$-finite.
\end{lemma}

\begin{proof}

  Assume for a contradiction that $\tp(d/M')$ is $M$-finite.  Let $\bar
  b=(b_1,\ldots,b_m)\in M^m$ witness this, of minimal length (as a tuple).

For any $a\in M'$, we can find $f(\bar b)$, with $f$
$\emptyset$-definable, such that $f(\bar b)$ lies between $d$ and $a$.
Considering $d\restrict i$, for $i<\lambda$, we can find $\{f_i(\bar
b)\}_{i<\lambda}$ with $f_i(\bar b)\restrict i = d\restrict i$.

By quantifier elimination for divisible ordered abelian groups, we know that
each $f_i(\bar b)$ is an affine linear combination (with rational coefficients)
of the $b_j$'s, with the affine part given by $c\in Q^{<\lambda}$.  If we take
$\alpha=\max(\lh(b_j)\mid j\le m)$, then for any $\beta$, $f_\beta(\bar b)$ can
have no components in $Q'$ past the $\alpha$th one.  But this is clearly
impossible.

\end{proof}

This completes our proof that $N$ is a maximal small extension of $M$.  $N$ is
certainly a proper extension of $M$, since any element with cofinal components
in $Q'$ can be adjoined to form $M_1$.  It remains to determine its size.  We
lose nothing by restricting to the case where $\lambda$ is an infinite cardinal.
We consider $G^\lambda$ and its divisible subgroups as $\QQ$-vector spaces.

\begin{claim}\label{betaclaim}
Let $\beta$ be a cardinal such that there exist linearly independent
$\{a_i\}_{i<\beta}\in G^\lambda$ with $(M+
Q^\lambda)\oplus\Span(\{a_i\}_{i<\beta})=G^\lambda$.  Then in the construction
above we can ensure that $|N|=\beta+|M|$.
\end{claim}
\begin{proof}
We show the claim by showing that we can take $a_i$ to be the element adjoined
to $M_i$ to produce $M_{i+1}$, for every $i<\beta$.  For any such $i$, the
element $a_i$ has the property that every element of $\dcl(Ma_{\le
  i})\setminus\dcl(Ma_{<i})$ has cofinal components in $Q'$, since otherwise that
element would be in the span of $M+Q^\lambda+\Span(\{a_j\}_{j<i}\})$, and
$\Span(\{a_j\}_{j<\beta})$ is linearly independent from $M+Q^\lambda$.
\end{proof}

Let $W$ be a divisible subgroup of $G$ such that $G=Q\oplus W$.  Then
$G^\lambda=Q^\lambda\oplus W^\lambda$.  Let $\gamma=|W|$.

\begin{claim}
$\dim W^\lambda=\gamma^\lambda$.
\end{claim}
\begin{proof}
Note that $|W^\lambda|\ge\gamma^\lambda\ge 2^{\aleph_0}$, since any element can
be uniquely written as a $\lambda$-sequence of elements of $W$.  Since we are
considering $W^\lambda$ as a vector space over $\QQ$, a countable field, it
follows that $\dim(W^\lambda)=|W^\lambda|$.
\end{proof}

We can write $W^\lambda=W^{<\lambda}\oplus X$, for some divisible subgroup $X$
of $W^\lambda$.

\begin{claim}\label{xworksclaim}
$(M+Q^\lambda)\oplus X=G^\lambda$.
\end{claim}
\begin{proof}
\begin{multline*}
G^\lambda=Q^\lambda\oplus W^\lambda=(Q^\lambda\oplus W^{<\lambda})\oplus
  X\\=((Q^{<\lambda}+W^{<\lambda})+Q^\lambda)\oplus
  X=(G^{<\lambda}+Q^\lambda)\oplus X=(M+Q^\lambda)\oplus X.
\end{multline*}
\end{proof}

This implies that we may let the desired sequence $\{a_i\}_{i<\beta}$ be given
by a basis for $X$.

\begin{claim}\label{xdimclaim}
$\dim X=\gamma^\lambda$.
\end{claim}

\begin{proof}
We construct a set of independent (even over $W^{<\lambda}$) elements of
$W^\lambda$, with size $\gamma^\lambda$ and each element of length $\lambda$,
showing that $\dim X\ge\gamma^\lambda$, which is enough.

Since $\lambda\times\lambda=\lambda$, we can find $\lambda$ disjoint subsets of
$\lambda$ of length $\lambda$ (necessarily cofinal).  Let $\{X_i\mid
i<\lambda\}$ be the characteristic functions of these subsets -- each $X_i$ is a
binary sequence of length $\lambda$.  For $b\in W$, let $bX_i$ denote the
element of $W^\lambda$ obtained by replacing each $1$ in the sequence $X_i$ by
$b$.

For $f\in W^\lambda$, let $A_f=\sum_{i<\lambda} f(i)X_i$.  This sum is
well-defined, because no two $X_i$s are non-zero on the same component.  We know
that there is a basis of $W^\lambda$ of size $\gamma^\lambda$, say
$\{f_j\}_{j<\gamma^\lambda}$.  Denote $A_{f_j}$ by $A_j$.  We show that
$\{A_j\mid j<\gamma^\lambda\}$ is linearly independent and its span is disjoint
from $W^{<\lambda}\setminus\{0\}$.  Without loss of generality, it is enough to
show that no non-zero linear combination of $A_1,\ldots,A_n$ is in
$W^{<\lambda}$.

Suppose that $q_1A_1+\ldots+q_nA_n=c$, where $q_j\in\QQ$, $c\in W^{<\lambda}$.
This then implies that $\sum_{i<\lambda} (\sum_{j\le n} q_jf_j(i))X_i=c$.  Fix
$i<\lambda$.  If $k,l\in X_i$, then it is clear that the left-hand side has the
same value at its $k$ and $l$ components.  But if we choose $k<\lh(c)<l$, then
the $l$th component must be $0$, so the $k$th component is too.  Since this
holds for every $i<\lambda$ and $k<\lh(c)$, this implies $c=0$.  But this
implies $\sum_{j\le n} q_jf_j=0$, and so $q_j=0$, $j\le n$, and hence the $A_i$s
are linearly independent.

\end{proof}

Claim \ref{betaclaim} applied to a basis of $X$, using Claims \ref{xworksclaim}
and \ref{xdimclaim}, shows that $|N|=|M|+\gamma^\lambda$, where $\gamma$ is the
cardinality of $W$, a divisible subgroup of $G$ with $Q\oplus W=G$.  The next
section shows how this implies the cardinality statements of the proposition.
\end{proof}

\section{Cardinalities}
When $\lambda=\aleph_0$, $G$ is the real closure of $\QQ$, and $Q=\QQ$, then
$|M|=\aleph_0$ and $|N|=2^{\aleph_0}$: the bound is as sharp as possible.  In
general, for any $\alpha$ an elementary compactness argument shows there exist
$G$ and $Q$ such that $|G|=|W|=\alpha$.  If we take $\lambda$ to be $\aleph_0$,
then $|M|=\alpha$.  However, while $N$ exists, it is possible that $|N|=|M|$,
since $\alpha^{\aleph_0}$ may be $\alpha$.  Note, though, that $N$ is still a
proper maximal small extension of $M$, examples of which were not known before.

$M$ can have any cardinality of the form $\alpha^{<\lambda}$ for any two
infinite cardinals $\alpha,\lambda$.  Then $N$ can have cardinality at least
$\alpha^\lambda$.  This corresponds to the full tree of height $\lambda$ with
$\alpha$-many branchings at each node.  Note that the definition of $\Ded(-)$
can be rephrased in terms of trees, so that in fact $\Ded(\beta)$ is the $\sup$
of the cardinalities of the completions of trees of cardinality $\beta$.  It is
clear that if a cardinal $\beta$ is of the form $\alpha^{<\lambda}$, then
$\Ded(\beta)=\alpha^\lambda$.  Thus, for such cardinals $\beta$, this
construction shows that the maximal small extension has cardinality
$\Ded(\beta)$.

In general, it is not hard to see that the construction above can be done with
$G^\lambda$ replaced by $\prod_{i<\lambda} G_i$, where each $G_i$ is a divisible
ordered abelian group, $Q^\lambda$ replaced by $\prod_{i<\lambda} Q_i$, where
each $Q_i$ is a dense divisible subgroup of $G_i$, and $M$ the subgroup
consisting of all elements of $N$ of length $<\lambda$.  It is an open question,
potentially independent from ZFC, whether such a configuration can actually
witness the Dedekind number for every cardinal.  It would not if there were some
cardinal, $\alpha$, with $\alpha\ne |M|$ for any such $M$, but
$\Ded(\alpha)>|N|$ for any corresponding $N$.  That would require that
$\Ded(\alpha)$ be witnessed by a highly asymmetric tree, but no results are
known on the types of trees that are needed to witness $\Ded$.

\bibliography{../janak}{}

\begin{thebibliography}{Kud08}

\bibitem[HS91]{HrSh91}
Ehud Hrushovski and Saharon Shelah.
\newblock Stability and omitting types.
\newblock {\em Israel J. Math.}, 74(2-3):289--321, 1991.

\bibitem[Kud08]{Kudaibergenov8}
K.~Zh. Kudaibergenov.
\newblock Small extensions of models of o-minimal theories and absolute
  homogeneity.
\newblock {\em Siberian Advances in Mathematics}, 18(2):118--123, June 2008.

\bibitem[Mar86]{Marker86}
David Marker.
\newblock Omitting types in o-minimal theories.
\newblock {\em J. Symbolic Logic}, 51(1):63--74, March 1986.

\end{thebibliography}
\bibliographystyle{halpha}







\end{document}